\pdfoutput=1
\documentclass[a4paper,11pt]{amsart}

\usepackage[margin=1in,includehead,includefoot]{geometry}
\usepackage[utf8]{inputenc}
\usepackage[T1]{fontenc}
\usepackage{lmodern,microtype}
\usepackage{mathtools,amsthm,amssymb}
\usepackage{enumitem}
\usepackage{tikz}
\usepackage{hyperref,bookmark}

\hypersetup{colorlinks=true,linkcolor=black,citecolor=black,filecolor=black,urlcolor=black}

\makeatletter
\def\@makefntext{\setlength{\parindent}{0pt}\@makefnmark}
\let\citationorig\citation
\def\citation#1{\citationorig{#1}\@for\@tempa:=#1\do{\@ifundefined{cit@\@tempa}{\global\@namedef{cit@\@tempa}{}}{}}}
\let\bibitemorig\bibitem
\def\bibitem#1{\@ifundefined{cit@#1}{\typeout{LaTeX Warning: Unused bibitem `#1'}}{}\bibitemorig{#1}}
\let\old@setaddresses\@setaddresses
\def\@setaddresses{\bigskip{\parindent 0pt\let\scshape\relax\let\ttfamily\relax\old@setaddresses}}

\makeatother

\newtheorem{theorem}{Theorem}[section]
\newtheorem{lemma}[theorem]{Lemma}
\newtheorem{proposition}[theorem]{Proposition}
\newtheorem{observation}[theorem]{Observation}
\newtheorem{conjecture}[theorem]{Conjecture}
\theoremstyle{definition}
\newtheorem{question}[theorem]{Question}

\renewenvironment{enumerate}{\begin{enumorig}[label={\upshape(\arabic*)}, noitemsep, topsep=3pt plus 3pt, leftmargin=*]}{\end{enumorig}}

\renewenvironment{itemize}{\begin{itemorig}[label=\textbullet, noitemsep, topsep=3pt plus 3pt, labelsep=.6em, labelindent=.2em, leftmargin=*]}{\end{itemorig}}

\newcommand{\bfk}{\mathbf{k}}
\newcommand{\bfone}{\mathbf{1}}
\newcommand{\bftwo}{\mathbf{2}}
\newcommand{\bfthree}{\mathbf{3}}
\newcommand{\cgD}{\mathcal{D}}

\newcommand{\cgR}{\mathcal{R}}
\newcommand{\cgX}{\mathcal{X}}
\newcommand{\lex}{\mathrm{lex}}
\newcommand{\setR}{\mathbb{R}}

\DeclareMathOperator{\Inc}{Inc}
\DeclareMathOperator{\Dn}{Dn}
\DeclareMathOperator{\Up}{Up}
\DeclareMathOperator{\Ar}{Ar}
\DeclareMathOperator{\dn}{dn}
\DeclareMathOperator{\up}{up}
\DeclareMathOperator{\BT}{BT}
\DeclareMathOperator{\Spec}{Spec}

\newcommand{\length}[1]{{\lVert #1\rVert}}
\newcommand{\downset}{\mathord{\downarrow}}

\let\leq\leqslant
\let\geq\geqslant

\allowdisplaybreaks

\hypersetup{pdftitle={Dimension of posets with planar cover graphs excluding two long incomparable chains},
            pdfauthor={David M. Howard, Noah Streib, William T. Trotter, Bartosz Walczak, Ruidong Wang}}

\title[Dimension of posets with planar cover graphs excl.\ two long incomp.\ chains]{Dimension of posets with planar cover graphs\\excluding two long incomparable chains}

\author[D.~M. Howard\and N. Streib\and W.~T. Trotter\and B. Walczak\and R. Wang]{David~M. Howard\and Noah Streib\and William~T. Trotter\and Bartosz Walczak\and Ruidong Wang}

\address[Noah Streib]{Center for Computing Sciences, 17100 Science Dr., Bowie, MD 20715, USA}
\email{\href{mailto:nsstrei@super.org}{nsstrei@super.org}}

\address[William~T. Trotter]{School of Mathematics, Georgia Institute of Technology, Atlanta, GA 30332, USA}
\email{\href{mailto:trotter@math.gatech.edu}{trotter@math.gatech.edu}}

\address[Bartosz Walczak]{Department of Theoretical Computer Science, Faculty of Mathematics and Computer Science, Jagiellonian University, Kraków, Poland}
\email{\href{mailto:walczak@tcs.uj.edu.pl}{walczak@tcs.uj.edu.pl}}

\address[Ruidong Wang]{Blizzard Entertainment, Irvine, CA 92618, USA}
\email{\href{mailto:ruwang@blizzard.com}{ruwang@blizzard.com}}

\thanks{A journal version of this paper appeared in \href{https://doi.org/10.1016/j.jcta.2018.11.016}{\emph{J.\ Comb.\ Theory Ser.~A} 164, 1--23, 2019}.}

\thanks{Bartosz Walczak was partially supported by National Science Center of Poland grant 2015/18/E/ST6/00299.}

\begin{document}

\begin{abstract}
It has been known for more than 40 years that there are posets with planar cover graphs and arbitrarily large dimension.
Recently, Streib and Trotter proved that such posets must have large height.
In fact, all known constructions of such posets have two large disjoint chains with all points in one chain incomparable with all points in the other.
Gutowski and Krawczyk conjectured that this feature is necessary.
More formally, they conjectured that for every $k\geq 1$, there is a constant $d$ such that if $P$ is a poset with a planar cover graph and $P$ excludes $\bfk+\bfk$, then $\dim(P)\leq d$.
We settle their conjecture in the affirmative.
We also discuss possibilities of generalizing the result by relaxing the condition that the cover graph is planar.
\end{abstract}

\maketitle

\section{Introduction}

We assume that the reader is familiar with basic notation and terminology for posets, including subposets, chains and antichains, minimal and maximal elements, linear extensions, order diagrams, and cover graphs.
Extensive background information on the combinatorics of posets can be found in~\cite{Tro-book,Tro95}.
We will also assume that the reader is familiar with basic concepts of graph theory, including subgraphs, induced subgraphs, paths and cycles, and planar graphs.

A subposet $Q$ of $P$ is \emph{convex} if $y\in Q$ whenever $x,z\in Q$ and $x<y<z$ in $P$.
When $Q$ is a convex subposet of $P$, the cover graph of $Q$ is an induced subgraph of the cover graph of $P$.
Traditionally, the elements of a poset are called \emph{points}, and this is what we do in this paper.

Dushnik and Miller~\cite{DM41} defined the \emph{dimension} of a poset $P$, denoted by $\dim(P)$, as the least positive integer $d$ for which there is a family $\cgR=\{L_1,\ldots,L_d\}$ of linear extensions of $P$ such that $x\leq y$ in $P$ if and only if $x\leq y$ in all $L_1,\ldots,L_d$.
Clearly, if $Q$ is a subposet of $P$, then $\dim(Q)\leq\dim(P)$.
A poset has dimension $1$ if and only it is a chain.

For $d\geq 2$, the \emph{standard example} $S_d$ is the poset of height $2$ consisting of $d$ minimal elements $a_1,\ldots,a_d$ and $d$ maximal elements $b_1,\ldots,b_d$ with $a_i<b_j$ in $S_d$ if and only if $i\neq j$.
As noted in~\cite{DM41}, $\dim(S_d)=d$ for every $d\geq 2$.
So every poset that contains a large standard example has large dimension.
On the other hand, it is well known that there are posets that have large dimension but do not contain the standard example $S_2$ (see the more comprehensive discussion in~\cite{BHPT16}).

In recent years, there have been a series of research papers exploring connections between the dimension of a poset $P$ and graph-theoretic properties of the cover graph of $P$.
This paper continues with that theme.
A poset $P$ is \emph{planar} if it has a drawing with no edge crossings in its order diagram.
All planar posets have planar cover graphs, and it is well known that there are non-planar posets with planar cover graphs (see~\cite{Tro-book}, page~67).

It is an easy exercise to show that the standard example $S_d$ is a planar poset when $2\leq d\leq 4$, while the cover graph of $S_d$ is non-planar when $d\geq 5$.
However, in~\cite{Tro78}, it is shown that for every $d\geq 5$, the non-planar poset $S_d$ is a subposet of a poset with a planar cover graph.
Subsequently, Kelly~\cite{Kel81} proved the stronger result: for every $d\geq 5$, the non-planar poset $S_d$ is a subposet of a planar poset $P$ with $\dim(P)=d$ (see Figure~\ref{fig:kelly}).

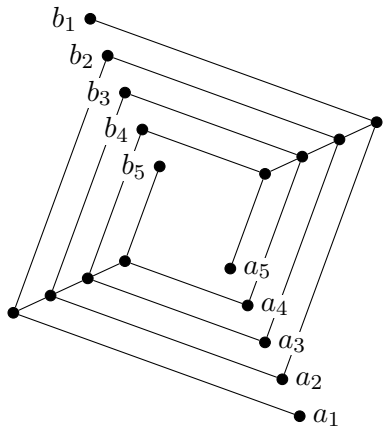
\begin{figure}[t]
\begin{tikzpicture}[scale=.54,rotate=25,baseline=(current bounding box.center)]
\tikzstyle{every node}=[circle,draw,fill,minimum size=4pt,inner sep=0pt]
\tikzstyle{every label}=[rectangle,draw=none,fill=white,inner sep=1pt,label distance=1.5pt]
\node (z1) at (-5,-0.25) {};
\node (z2) at (-4,-0.25) {};
\node (z3) at (-3,-0.25) {};
\node (z4) at (-2,-0.25) {};
\node (w1) at (5,0.25) {};
\node (w2) at (4,0.25) {};
\node (w3) at (3,0.25) {};
\node (w4) at (2,0.25) {};
\node[label={[yshift=-0.5pt]right:$a_1$}] (a1) at (0.25,-5.5) {};
\node[label=left:$b_1$] (b1) at (-0.25,5.5) {};
\node[label={[yshift=-0.5pt]right:$a_2$}] (a2) at (0.25,-4.5) {};
\node[label=left:$b_2$] (b2) at (-0.25,4.5) {};
\path (z1) edge (a1) edge (b2) edge (z2);
\path (w1) edge (b1) edge (a2) edge (w2);
\node[label={[yshift=-0.5pt]right:$a_3$}] (a3) at (0.25,-3.5) {};
\node[label=left:$b_3$] (b3) at (-0.25,3.5) {};
\path (z2) edge (a2) edge (b3) edge (z3);
\path (w2) edge (b2) edge (a3) edge (w3);
\node[label={[yshift=-0.5pt]right:$a_4$}] (a4) at (0.25,-2.5) {};
\node[label=left:$b_4$] (b4) at (-0.25,2.5) {};
\path (z3) edge (a3) edge (b4) edge (z4);
\path (w3) edge (b3) edge (a4) edge (w4);
\node[label={[yshift=-0.5pt]right:$a_5$}] (a5) at (0.25,-1.5) {};
\node[label=left:$b_5$] (b5) at (-0.25,1.5) {};
\path (z4) edge (a4) edge (b5);
\path (w4) edge (b4) edge (a5);
\end{tikzpicture}
\caption{Kelly's example of a planar poset containing the standard example $S_5$ as a subposet}
\label{fig:kelly}
\end{figure}

In this paper, we do not distinguish between isomorphic posets, and we say that $P$ \emph{contains} $Q$ when there is a subposet of $P$ that is isomorphic to $Q$.
Also, we say $P$ \emph{excludes} $Q$ when $P$ does not contain $Q$.
For a positive integer $k$, we let $\bfk$ denote a $k$-element chain, and we let $\bfk+\bfk$ denote a poset consisting of two chains of size $k$ with all points in one chain incomparable with all points in the other.
The above-mentioned constructions of posets with planar cover graphs and arbitrarily large dimension raise the following questions.

\begin{question}
Which of the following statements are true for every poset $P$ with a planar cover graph and sufficiently large dimension?
\begin{enumerate}
\item $P$ has many minimal elements;
\item $P$ has large height, that is, $P$ contains $\bfk$ for some large
value of $k$;
\item $P$ contains $\bfk+\bfk$ for some large value of $k$;
\item $P$ contains $S_k$ for some large value of $k$.
\end{enumerate}
\end{question}

The construction in~\cite{Tro78} shows that for every $d\geq 2$, there is a poset with dimension $d$, a unique minimal element, a unique maximal element, and a planar cover graph.
On the other hand, in~\cite{TW16}, the following result is proved for planar posets.

\begin{theorem}
If\/ $P$ is a planar poset with\/ $t$ minimal elements, then\/ $\dim(P)\leq 2t+1$.
\end{theorem}

Furthermore, it is shown in~\cite{TW16} that this inequality is tight when $t=1$ and $t=2$.
However, when $t\geq 3$, it is only known that there are planar posets with $t$ minimal elements that have dimension $t+3$.
Since a poset and its dual have the same dimension, entirely analogous statements can be made about maximal elements.

The second question was answered in the affirmative in~\cite{ST14}, where the following theorem (restated in a form consistent with the results of this paper) is proved.

\begin{theorem}
\label{thm:streib-trotter}
For every positive integer\/ $k$, there is an integer\/ $d$ such that if\/ $P$ is a poset that excludes\/ $\bfk$ and the cover graph of\/ $P$ is planar, then\/ $\dim(P)\leq d$.
\end{theorem}

The bound on $d$ from~\cite{ST14} is very weak, due to extensive use of Ramsey theory in the proof; however, greatly improved bounds are available via~\cite{MW17}.
Furthermore, it is shown in~\cite{JMW17} that planar posets excluding $\bfk$ have dimension bounded by $O(k)$.

Gutowski and Krawczyk \cite{GK-personal} posed the third question and conjectured that it should also have an affirmative answer.
In this paper, we will settle their conjecture in the affirmative by proving the following theorem, which is the main result of this paper.

\begin{theorem}
\label{thm:two-chains}
For every positive integer\/ $k$, there is an integer\/ $d$ such that if\/ $P$ is a poset that excludes\/ $\bfk+\bfk$ and the cover graph of\/ $P$ is planar, then\/ $\dim(P)\leq d$.
\end{theorem}

While the conjecture of Gutowski and Krawczyk might seem entirely natural just from reflecting on the properties of the Kelly construction, it was also motivated by the results of \cite{BKS10,DJW12,FKT13,JM11,LMS+14}, where combinatorial properties of posets excluding $\bfk+\bfk$ played a central role.

The fourth question, which was apparently first raised in~\cite{Tro-book} (see the comment on page~119), remains open, and we consider it one of the central challenges in this area of research.
Most researchers feel that the answer is again ``yes''.
Formally, we can state the following conjecture.

\begin{conjecture}
\label{con:Sk}
For every positive integer\/ $k$, there exists an integer\/ $d$ such that if\/ $P$ is a poset that excludes the standard example\/ $S_k$ and the cover graph of\/ $P$ is planar, then\/ $\dim(P)\leq d$.
\end{conjecture}

In the next section, we provide a brief summary of notation, terminology, and background material.
This discussion applies to any research problem involving dimension.
Then, in Section~\ref{sec:two-chains}, we develop some properties of the class of posets that exclude $\bfk+\bfk$.
As these results may find application to other combinatorial problems for posets, the results of that section are presented for posets in general---with no assumption that the cover graph is planar.
The proof of our main theorem is given in the next three sections.
Finally, in Section~\ref{sec:connections}, we discuss possibilities of generalizing Theorem~\ref{thm:two-chains} and Conjecture~\ref{con:Sk} beyond planarity---to posets that have cover graphs with excluded minors and excluded topological minors.

\section{Notation, terminology, and background material}

Let $P$ be a poset.
A family $\cgR=\{L_1,\ldots,L_d\}$ of linear extensions of $P$ is called a \emph{realizer} of $P$ when the following holds: $x\leq y$ in $P$ if and only if $x\leq y$ in all $L_1,\ldots,L_d$.
Thus $\dim(P)$ is the least positive integer $d$ such that $P$ has a realizer of size $d$.
Accordingly, to establish an upper bound of the form $\dim(P)\leq d$, the most natural approach is simply to construct a realizer of size $d$ for $P$.
However, in recent papers \cite{FT00,FTW15,JMM+16,JMOW-arxiv,JMT+17,JMW17,JMW18,MW17,ST14,TWW18,Wal17}, another approach has been taken.
Let $\Inc(P)$ denote the set of ordered incomparable pairs of $P$.
Clearly, a family $\cgR$ of linear extensions of $P$ is a realizer of $P$ if and only if for every $(x,y)\in\Inc(P)$, there is $L\in\cgR$ with $x>y$ in $L$.
In this case, we say that $L$ \emph{reverses} the incomparable pair $(x,y)$.
More generally, when $S$ is a set of incomparable pairs of $P$, a linear extension $L$ \emph{reverses} $S$ when $x>y$ in $L$ for every $(x,y)\in S$.
A set $S\subseteq\Inc(P)$ is \emph{reversible} when there is a linear extension $L$ of $P$ that reverses $S$, and a family $\cgR$ of linear extensions \emph{reverses} $S$ when for every $(x,y)\in S$, there is $L\in\cgR$ that reverses $(x,y)$.
With these ideas in hand, when $S\subseteq\Inc(P)$, we can define the \emph{dimension} of $S$, denoted by $\dim(S)$, as the least positive integer $d$ for which there is a family of $d$ linear extensions of $P$ that reverses $S$.
Clearly, $\dim(P)=\dim(\Inc(P))$, so we can also say that $\dim(P)$ is the least positive integer $d$ for which there is a partition of $\Inc(P)$ into $d$ reversible sets.

An indexed family $\{(x_\alpha,y_\alpha)\}_{\alpha=1}^s$ of incomparable pairs of $P$ with $s\geq 2$ is called an \emph{alternating cycle} when $x_\alpha\leq y_{\alpha+1}$ in $P$ for every index $\alpha$ considered cyclically in $\{1,\ldots,s\}$ (that is, $y_{s+1}=y_1$).
An alternating cycle is \emph{strict} when there are no other comparabilities, that is, $x_i\leq y_j$ in $P$ if and only if $j\equiv i+1\pmod{s}$.
The following elementary lemma, proved in~\cite{TM77}, provides a convenient test to determine whether a subset of $\Inc(P)$ is reversible.

\begin{lemma}
If\/ $P$ is a poset and\/ $S\subseteq\Inc(P)$, then the following statements are equivalent:
\begin{enumerate}
\item $S$ is not reversible;
\item $S$ contains an alternating cycle;
\item $S$ contains a strict alternating cycle.
\end{enumerate}
\end{lemma}

A typical approach to show that the set $\Inc(P)$ can be partitioned into $d$ reversible sets is by defining a $d$-coloring of the pairs in $\Inc(P)$ with the property that no (strict) alternating cycle is monochromatic.
However, the rules for assigning colors can be quite complicated, and that will certainly be the case here.

\section{Posets that exclude two long incomparable chains}
\label{sec:two-chains}

In this section, we present some general considerations on posets excluding two long incomparable chains.
If a poset $P$ excludes $\bfone+\bfone$, then $P$ is a chain, so $\dim(P)=1$.
For the rest of this section, we fix an integer $k\geq 2$ and a poset $P$ that excludes $\bfk+\bfk$.
We make no assumption on the structure of the cover graph of $P$.

Let $h$ denote the height of $P$, and let $C=\{c_1<\cdots<c_h\}$ be a chain in $P$ of size $h$.
For each point $z\in P-C$, define integers $\dn(z)$ and $\up(z)$ as follows:
\begin{align*}
\dn(z)&=\begin{cases}
0&\text{if $z$ is incomparable with $c_1$ in $P$,}\\
i&\text{otherwise, where $i$ is greatest in $\{1,\ldots,h\}$ such that $z>c_i$ in $P$;}
\end{cases}\\
\up(z)&=\begin{cases}
h+1&\text{if $z$ is incomparable with $c_h$ in $P$,}\\
j&\text{otherwise, where $j$ is least in $\{1,\ldots,h\}$ such that $z<c_j$ in $P$.}
\end{cases}
\end{align*}
Note that $0\leq\dn(z)\leq h-1$ and $2\leq\up(z)\leq h+1$ for every point $z\in P-C$, by maximality of the chain $C$.
Define
\begin{alignat*}{2}
\Dn(i)&=\{z\in P-C\colon\dn(z)=i\}&\quad&\text{for $0\leq i\leq h-1$,}\\
\Up(j)&=\{z\in P-C\colon\up(z)=j\}&\quad&\text{for $2\leq j\leq h+1$.}
\end{alignat*}

\begin{lemma}
\label{lem:Dn_iUp_j}
For\/ $0\leq i\leq h-1$, the subposet\/ $\Dn(i)$ of\/ $P$ is convex and has height at most\/ $2k-2$.
More generally, for\/ $0\leq i\leq i+m\leq h-1$, the subposet\/ $\Dn(i,m)$ of\/ $P$ defined by
\[\Dn(i,m)=\bigcup_{\alpha=i}^{i+m}\Dn(\alpha)\cup\{c_{i+1},\ldots,c_{i+m}\}\]
is convex and has height at most\/ $m+2k-2$.

Dually, for\/ $2\leq j\leq h+1$, the subposet\/ $\Up(j)$ of\/ $P$ is convex and has height at most\/ $2k-2$.
More generally, for\/ $2\leq j\leq j+m\leq h+1$, the subposet\/ $\Up(j,m)$ of\/ $P$ defined by
\[\Up(j,m)=\bigcup_{\alpha=j}^{j+m}\Up(\alpha)\cup\{c_j,\ldots,c_{j+m-1}\}\]
is convex and has height at most\/ $m+2k-2$.
\end{lemma}

\begin{proof}
We only show the proof of the first part, as the second is dual.
It is clear that the subposet $\Dn(i,m)$ is convex.
The fact that $C$ is a maximum chain implies that the height of $\Dn(i,m)$ is at most $h-i$, so the desired inequality follows if $h-i\leq m+2k-2$.
Suppose $h\geq i+m+2k-1\geq i+m+k$.
Let $Q=\{z\in\Dn(i,m)\colon\up(z)\leq c_{i+m+k}\}$.
The fact that $C$ is a maximum chain forces the height of $Q$ to be at most $m+k-1$.
Furthermore, the height of the subposet $\Dn(i,m)-Q$ is at most $k-1$, because all points of $\Dn(i,m)-Q$ are incomparable with the $k$-element chain $\{c_{i+m+1}<\cdots<c_{i+m+k}\}$.
Hence the height of $\Dn(i,m)$ is at most $(m+k-1)+(k-1)=m+2k-2$.
\end{proof}

For $0\leq i\leq h$, define $\Ar(i,i+1)=\{z\in P-C\colon\dn(z)\leq i$ and $\up(z)\geq i+1\}$.
Here, $\Ar$ stands for ``around''.

\begin{lemma}
\label{lem:Ar(i,i+1)}
For\/ $0\leq i\leq h$, the subposet\/ $\Ar(i,i+1)$ of\/ $P$ is convex and has height at most\/ $4k-4$.
\end{lemma}

\begin{proof}
It is clear that the subposet $\Ar(i,i+1)$ is convex.
Let $Q=\{z\in\Ar(i,i+1)\colon\up(z)\leq i+k\}$.
Thus $Q\subseteq\Up(i+1)\cup\cdots\cup\Up(i+k)\subseteq\Up(i+1,k-1)$.
It follows from Lemma~\ref{lem:Dn_iUp_j} that the height of $Q$ is at most $(k-1)+2k-2=3k-3$.
Furthermore, the height of the subposet $\Ar(i,i+1)-Q$ is at most $k-1$, because all points of $\Ar(i,i+1)-Q$ are incomparable with the $k$-element chain $\{c_{i+1}<\cdots<c_{i+k}\}$.
Hence the height of $\Ar(i,i+1)$ is at most $(3k-3)+(k-1)=4k-4$.
\end{proof}

\begin{lemma}
\label{lem:h(Q)}
Let\/ $z,w\in P-C$, $\dn(z)<\up(w)$, and\/ $w\leq z$ in\/ $P$.
If\/ $C'$ is a chain in\/ $P-C$ with\/ $w$ the least element and\/ $z$ the greatest element, then\/ $|C'|\leq 4k-4$.
\end{lemma}

\begin{proof}
We have $z,w\in\Ar(\dn(z),\dn(z)+1)$, which implies $C'\subseteq\Ar(\dn(z),\dn(z)+1)$.
We apply Lemma~\ref{lem:Ar(i,i+1)} to conclude that $|C'|\leq 4k-4$.
\end{proof}

Every incomparable pair $(x,y)$ of $P$ satisfies $\dn(y)<\up(x)$.
We call an incomparable pair $(x,y)$ of $P$ \emph{dangerous} if $\dn(x)<\dn(y)<\up(x)<\up(y)$ and \emph{safe} otherwise.

\begin{lemma}
\label{lem:safe}
If\/ $d_0$ is a positive integer such that every convex subposet\/ $Q$ of\/ $P$ of height at most\/ $2k-2$ satisfies\/ $\dim(Q)\leq d_0$, then there is a set of at most\/ $2d_0$ linear extensions of\/ $P$ that reverses all the safe incomparable pairs of\/ $P$.
\end{lemma}

\begin{proof}
It follows from Lemma~\ref{lem:Dn_iUp_j} that $\dim(\Dn(i))\leq d_0$ for $0\leq i\leq h-1$ and $\dim(\Up(j))\leq d_0$ for $2\leq j\leq h+1$.
First, consider $d_0$ linear extensions of $P$ that
\begin{itemize}
\item have block form $\Dn(0)<c_1<\Dn(1)<\cdots<c_{h-1}<\Dn(h-1)<c_h$,
\item induce $d_0$ linear extensions of $\Dn(i)$ witnessing $\dim(\Dn(i))\leq d_0$, for $0\leq i\leq h-1$.
\end{itemize}
By the first condition, these linear extensions reverse all incomparable pairs $(x,y)$ of $P$ such that $\dn(x)>\dn(y)$, and by the second condition---all incomparable pairs $(x,y)$ of $P$ such that $\dn(x)=\dn(y)$.
Then, consider $d_0$ more linear extensions of $P$ that
\begin{itemize}
\item have block form $c_1<\Up(2)<c_2<\cdots<\Up(h)<c_h<\Up(h+1)$,
\item induce $d_0$ linear extensions of $\Up(i)$ witnessing $\dim(\Up(i))\leq d_0$, for $2\leq i\leq h+1$.
\end{itemize}
By the first condition, these linear extensions reverse all incomparable pairs $(x,y)$ of $P$ such that $\up(x)>\up(y)$, and by the second condition---all incomparable pairs $(x,y)$ of $P$ such that $\up(x)=\up(y)$.
We conclude that an incomparable pair $(x,y)$ of $P$ is reversed by some of the $2d_0$ linear extensions unless $\dn(x)<\dn(y)<\up(x)<\up(y)$, that is, the pair $(x,y)$ is dangerous.
\end{proof}

In view of Lemma~\ref{lem:safe}, we can focus on reversing only the dangerous incomparable pairs when attempting for a bound on $\dim(P)$.
This is the starting point of the proof of Theorem~\ref{thm:two-chains} in the next sections.
We conclude this section with two results that will not be used further in the paper: one asserting that $\dim(P)=O(h)$ whenever the convex subposets of $P$ with bounded height have bounded dimension, and the other asserting that linear dependence on $h$ is necessary.

\begin{proposition}
If\/ $d_1$ is a positive integer such that every convex subposet\/ $Q$ of\/ $P$ of height at most\/ $4k-4$ satisfies\/ $\dim(Q)\leq d_1$, then\/ $\dim(P)\leq(h+1)d_1$.
\end{proposition}

\begin{proof}
We apply Lemma~\ref{lem:safe} to reverse all the safe incomparable pairs of $P$ using at most $2d_1$ linear extensions.
For every dangerous incomparable pair $(x,y)$ of $P$, we have $x,y\in\Ar(\dn(y),\dn(y)+1)$, where $1\leq\dn(y)\leq h-1$.
For $1\leq i\leq h-1$, it follows from Lemma~\ref{lem:Ar(i,i+1)} that $\dim(\Ar(i,i+1))\leq d_1$, and we can extend any linear extension of $\Ar(i,i+1)$ witnessing the dimension to a linear extension of $P$.
This way, we obtain a set of at most $(h-1)d_1$ linear extensions of $P$ reversing all the dangerous incomparable pairs of $P$, and the proof is complete.
\end{proof}

\begin{proposition}
For every positive integer\/ $n$, there is a poset\/ $P$ excluding\/ $\bfthree+\bfthree$ such that\/ $\dim(P)\geq n$, $h(P)=n+1$, and every convex subposet\/ $Q$ of\/ $P$ satisfies\/ $\dim(Q)\leq h(Q)+1$, where\/ $h(Q)$ denotes the height of\/ $Q$.
\end{proposition}

\begin{proof}
The poset $P$ consists of points $a_1,\ldots,a_n,b_1,\ldots,b_n,c_0,\ldots,c_n$ with the following cover relations: $c_0<\cdots<c_n$, $a_i<c_i$ and $c_{i-1}<b_i$ for $1\leq i\leq n$, and $a_i<b_j$ for $1\leq j<i\leq n$.
Every chain in $P$ of size at least $3$ contains a point of the form $c_i$, so $P$ excludes $\bfthree+\bfthree$.
The subposet of $P$ induced on $a_1,\ldots,a_n,b_1,\ldots,b_n$ is isomorphic to the standard example $S_n$, so $\dim(P)\geq n$.
Every convex subposet of $P$ of height $1$ is an antichain and therefore has dimension at most $2$.
Now, let $h\geq 2$, and let $Q$ be a convex subposet of $P$ of height $h$.
To prove $\dim(Q)\leq h+1$, we can assume without loss of generality that $Q$ is a maximal subposet of $P$ of height $h$.
It follows that $Q$ is comprised of points $a_{i+1},\ldots,a_n,b_1,\ldots,b_{i+h-1},c_i,\ldots,c_{i+h-1}$ for some $i\in\{0,\ldots,n-h+1\}$.
It is easy to check that the following $h+1$ linear extensions of $Q$ form a realizer of $Q$:
\begin{align*}
&a_{i+1}<\cdots<a_n<b_i<\cdots<b_1<c_i<b_{i+1}<c_{i+1}<\cdots<b_{i+h-1}<c_{i+h-1},\\
&c_i<a_{i+1}<c_{i+1}<\cdots<a_{i+j-1}<c_{i+j-1}<a_{i+j+1}<\cdots<a_n<b_{i+j}<a_{i+j}\\
&\hphantom{c_i}<b_1<\cdots<b_{i+j-1}<c_{i+j}<b_{i+j+1}<c_{i+j+1}<\cdots<b_{i+h-1}<c_{i+h-1}\quad\text{for }1\leq j<h,\\
&c_i<a_{i+1}<c_{i+1}<\cdots<a_{i+h-1}<c_{i+h-1}<a_n<\cdots<a_{i+h}<b_1<\cdots<b_{i+1}.\qedhere
\end{align*}
\end{proof}

\section{Proof of the main theorem}

For the proof of Theorem~\ref{thm:two-chains}, we fix a poset $P$ that excludes $\bfk+\bfk$, where $k\geq 2$, and has a planar cover graph, and we attempt to partition the set $\Inc(P)$ of incomparable pairs of $P$ into a bounded number of reversible subsets, where the bound depends only on $k$.
Following the notation and terminology of the preceding section, let $h$ be the height of $P$, and let $C=\{c_1<\cdots<c_h\}$ be a chain in $P$ of size $h$.
We use operators $\Dn$, $\Up$, and $\Ar$ as in the preceding section to denote appropriate convex subposets of $P-C$.
We use operators $\dn$ and $\up$ in a different way than in the preceding section, namely, to refer to points of the chain $C$ rather than integer numbers:
\begin{itemize}
\item for $z\in P$, we let $\dn(z)=c_i$ if $c_i\leq z$ and $i$ is greatest in $\{1,\ldots,h\}$ with this property;
\item for $z\in P$, we let $\up(z)=c_i$ if $z\leq c_i$ and $i$ is least in $\{1,\ldots,h\}$ with this property.
\end{itemize}
We can write $\dn(z)$ or $\up(z)$ only for points $z\in P$ for which the respective point in $C$ exists.

By Theorem~\ref{thm:streib-trotter}, there is an integer $d_1$ such that every poset $P'$ with a planar cover graph and with height at most $4k-4$ satisfies $\dim(P')\leq d_1$.
It follows that every convex subposet of $P$ of height at most $4k-4$ has dimension at most $d_1$.
For $1\leq i\leq h-1$, in particular, Lemma~\ref{lem:Ar(i,i+1)} yields $\dim(\Ar(i,i+1))\leq d_1$, so there is a coloring $\phi_i\colon\Inc(\Ar(i,i+1))\to\{1,\ldots,d_1\}$ such that for each color $\gamma\in\{1,\ldots,d_1\}$, the set of incomparable pairs of $\Ar(i,i+1)$ that are assigned color $\gamma$ by $\phi_i$ is reversible.
We fix the integer $d_1$ and the colorings $\phi_i$ for $1\leq i\leq h-1$ for this and the following sections.

By Lemma~\ref{lem:safe}, we can reverse all the safe incomparable pairs of $P$ using at most $2d_1$ linear extensions, and the remaining challenge is to reverse the dangerous incomparable pairs of $P$.
For every dangerous incomparable pair $(a,b)$ of $P$, we have $c_1<b$ and $a<c_h$ in $P$, so the points $\dn(b)$ and $\up(a)$ of the chain $C$ are defined.
For any set $S$ of dangerous incomparable pairs of $P$, let $A(S)$ denote the set of points $a\in P$ for which there is a point $b\in P$ with $(a,b)\in S$, and let $B(S)$ denote the set of points $b\in P$ for which there is a point $a\in P$ with $(a,b)\in S$.

Recall from the preceding section that for every dangerous incomparable pair $(a,b)$ of $P$, we have $a,b\in\Ar(i,i+1)\cap\cdots\cap\Ar(j-1,j)$, where $c_i=\dn(b)$ and $c_j=\up(a)$.
In particular, the dangerous incomparable pairs $(a,b)$ of $P$ with $\dn(b)=c_1$ belong to $\Inc(\Ar(1,2))$ and thus can be reversed using $d_1$ linear extensions.
Similarly, the dangerous incomparable pairs $(a,b)$ of $P$ with $\up(a)=c_h$ can be reversed using $d_1$ linear extensions.

Let $S_0$ be the set of dangerous incomparable pairs $(a,b)$ of $P$ with $\dn(b)>c_1$ and $\up(a)<c_h$ in $P$.
It remains to partition the set $S_0$ into a bounded number of reversible subsets.
By convention, we will write $S$ only to denote some subset of the set $S_0$, we will write $a$, $a'$, $a_\alpha$, etc.\ only to denote a point from $A(S)$, and we will write $b$, $b'$, $b_\alpha$, etc.\ only to denote a point from $B(S)$.

For a set $S\subseteq S_0$, we call a pair $(a,b)\in S$ \emph{left-safe} with respect to $S$ if there is no point $b'\in B(S)$ with $a\leq b'$ and $\dn(b')<\dn(b)$ in $P$.
The following lemma plays an important role in our argument.

\begin{lemma}
\label{lem:left-safe}
For a set\/ $S\subseteq S_0$, if\/ $S'$ is the set of pairs in\/ $S$ that are left-safe with respect to\/ $S$, then\/ $\dim(S')\leq d_1$.
\end{lemma}

\begin{proof}
For a color $\gamma\in\{1,\ldots,d_1\}$, let $S'(\gamma)$ be the subset of $S'$ consisting of all pairs $(a,b)\in S'$ such that $\phi_i(a,b)=\gamma$, where $c_i=\dn(b)$.
It is enough to prove that the set $S'(\gamma)$ is reversible for every $\gamma\in\{1,\ldots,d_1\}$.
Suppose to the contrary that the set $S'(\gamma)$ is not reversible for some $\gamma\in\{1,\ldots,d_1\}$.
This means that $S'(\gamma)$ contains an alternating cycle $\{(a_\alpha,b_\alpha)\}_{\alpha=1}^s$, where $s\geq 2$.
For every $\alpha\in\{1,\ldots,s\}$, we have $a_\alpha\leq b_{\alpha+1}$ and therefore $\dn(b_\alpha)\leq\dn(b_{\alpha+1})$ in $P$, because the pair $(a_\alpha,b_\alpha)$ is left-safe with respect to $S$.
Since the latter inequality holds for every $\alpha\in\{1,\ldots,s\}$, there is a point $c_i\in C$ with $1\leq i\leq h-1$ such that $c_i=\dn(b_\alpha)$ for every $\alpha\in\{1,\ldots,s\}$.
This implies that $\{(a_\alpha,b_\alpha)\}_{\alpha=1}^s$ is a monochromatic alternating cycle in $\Inc(\Ar(i,i+1))$, which is a contradiction.
\end{proof}

Let $G$ denote the cover graph of $P$, which is a planar graph.
We fix a plane straight-line drawing of $G$, that is, a drawing of $G$ in the plane using non-crossing straight-line segments for edges.
We assume, without loss of generality, that the least point $c_1$ of the maximum chain $C$ lies on the outer face of the drawing.

A \emph{witnessing path} for a pair $(x,y)$ with $x\leq y$ in $P$ is a path $u_0\cdots u_r$ in $G$ such that $x=u_0<\cdots<u_r=y$ in $P$ (in particular, $u_i$ is covered by $u_{i+1}$ in $P$ for $0\leq i\leq r-1$).
It is clear that every comparable pair $(x,y)$ with $x\leq y$ in $P$ has at least one witnessing path.
For the purpose of our proof, it is convenient to fix one witnessing path, to be denoted by $W(x,y)$, for each pair $(x,y)$ with $x\leq y$ in $P$.
However, we need to choose the paths $W(x,y)$ in a consistent way, which is achieved in the following lemma.

\begin{lemma}
There is a function\/ $(x,y)\mapsto W(x,y)$ that maps each pair\/ $(x,y)$ with\/ $x\leq y$ in\/ $P$ to a witnessing path\/ $W(x,y)$ for\/ $(x,y)$ in such a way that the following holds:
\begin{enumerate}
\item if\/ $x\leq c\leq y$ in\/ $P$ and\/ $c\in C$, then\/ $W(x,y)$ passes through\/ $c$; in particular, if\/ $\up(x)$ and\/ $\dn(y)$ are defined and\/ $\up(x)\leq\dn(y)$ in\/ $P$, then\/ $W(x,y)$ passes through\/ $\up(x)$ and\/ $\dn(y)$;
\item if\/ $x_1\leq x_2\leq y_2\leq y_1$ in\/ $P$ and\/ $W(x_1,y_1)$ passes through\/ $x_2$ and\/ $y_2$, then\/ $W(x_2,y_2)$ is the subpath of\/ $W(x_1,y_1)$ from\/ $x_2$ to\/ $y_2$.
\end{enumerate}
\end{lemma}

\begin{proof}
Assume some (arbitrary) total order $\prec$ on the points of $P$.
The order $\prec$ extends naturally to a lexicographic total order $\prec_\lex$ on finite sequences of points of $P$ as follows:
\begin{itemize}
\item the empty sequence is the least element in $\prec_\lex$;
\item for any two non-empty sequences $u_0\cdots u_r$ and $v_0\cdots v_s$ of points of $P$, we have $u_0\cdots u_r\prec_\lex v_0\cdots v_s$ if and only if $u_0<v_0$ or $u_0=v_0$ and $u_1\cdots u_r\prec_\lex v_1\cdots v_s$.
\end{itemize}
For every pair $(x,y)$ with $x\leq y$ in $P$, let $W(x,y)$ be the $\prec_\lex$-minimum witnessing path among all witnessing paths for $(x,y)$ passing through all points $c\in C$ with $x\leq c\leq y$ in $P$.
It is clear that the paths $W(x,y)$ so defined satisfy both conditions of the lemma.
\end{proof}

The \emph{length} of a witnessing path $W(x,y)$, denoted by $\length{W(x,y)}$, is the number of edges in $W(x,y)$.
Lemmas \ref{lem:Dn_iUp_j} and~\ref{lem:Ar(i,i+1)} imply the following bounds on the lengths of witnessing paths:
\begin{itemize}
\item $\length{W(\dn(z),z)}\leq 2k-2$ for any $z\in P$,
\item $\length{W(z,\up(z))}\leq 2k-2$ for any $z\in P$,
\item $\length{W(z,w)}\leq 4k-5$ for any $z,w\in P-C$ with $\dn(w)<\up(z)$.
\end{itemize}

For $i\in\{2,\ldots,h-1\}$, we classify every edge of the form $c_iz$ of $G$, where $z\notin C$, as a
\begin{itemize}
\item\emph{left-edge} if the edges $c_ic_{i-1}$, $c_iz$, $c_ic_{i+1}$ occur in this order clockwise around $c_i$,
\item\emph{right-edge} if the edges $c_ic_{i-1}$, $c_iz$, $c_ic_{i+1}$ occur in this order counterclockwise around $c_i$.
\end{itemize}
Since $c_2\leq\dn(b)<\up(a)\leq c_{h-1}$ in $P$ for any $(a,b)\in S_0$, the above yields a partition of $S_0$ into four classes $S_{LL}$, $S_{LR}$, $S_{RL}$, and $S_{RR}$ according to how the last edge of $W(a,\up(a))$ and the first edge of $W(\dn(b),b)$ are classified.
That is, for every $(a,b)\in S_0$, we have
\begin{itemize}
\item $(a,b)\in S_{LL}\cup S_{LR}$ if the last edge of $W(a,\up(a))$ is a left-edge,
\item $(a,b)\in S_{RL}\cup S_{RR}$ if the last edge of $W(a,\up(a))$ is a right-edge,
\item $(a,b)\in S_{LL}\cup S_{RL}$ if the first edge of $W(\dn(b),b)$ is a left-edge,
\item $(a,b)\in S_{LR}\cup S_{RR}$ if the first edge of $W(\dn(b),b)$ is a right-edge.
\end{itemize}
To complete the proof of Theorem~\ref{thm:two-chains}, we will establish the following:
\begin{itemize}
\item $\dim(S_{LL})=O(k^3d_1)$ and $\dim(S_{RR})=O(k^3d_1)$, in the next section;
\item $\dim(S_{LR})=O(k^3+k^2d_1)$ and $\dim(S_{RL})=O(k^3+k^2d_1)$, in Section \ref{sec:opposite-side}.
\end{itemize}
This allows us to conclude that $\dim(S_0)=O(k^3d_1)$ and thus $\dim(P)=O(k^3d_1)$.

\section{Same-side dangerous pairs}
\label{sec:same-side}

In this section, we show that $\dim(S_{LL})=O(k^3d_1)$ and $\dim(S_{RR})=O(k^3d_1)$.
We present the argument only for $S_{LL}$, and the argument for $S_{RR}$ is symmetric.
To simplify the notation used in this portion of the proof, we (temporarily) set $S=S_{LL}$.

Recall that $c_2<b$ in $P$ for every $b\in B(S)$, as $S$ contains only dangerous pairs with $\dn(b)>c_1$ in $P$.
Furthermore, the choice of the witnessing paths guarantees that
\begin{itemize}
\item for any $b\in B(S)$, the common part of $W(c_1,b)$ with $W(c_1,c_h)$ is $W(c_1,\dn(b))$,
\item for any $b,b'\in B(S)$, the common part of $W(c_1,b)$ and $W(c_1,b')$ is $W(c_1,z)$ for some $z\in P$.
\end{itemize}
Therefore, the union of $W(c_1,c_h)$ and the witnessing paths $W(c_1,b)$ over all $b\in B(S)$ forms a tree, which we call the \emph{basic tree} and denote by $\BT$.
The points, edges, and paths in $\BT$ are called \emph{basic points}, \emph{basic edges}, and \emph{basic paths}.
For any two basic points $u$ and $v$, let $\BT(u,v)$ denote the unique path in $\BT$ (made of basic points and basic edges) between $u$ and $v$.

Let $\prec$ be the total order on the basic points determined by carrying out a depth-first search of $\BT$ from $c_1$ that processes the children of each node in the clockwise order and sets $u\prec v$ when $u$ is visited before $v$ for the first time.
In other words, for any two basic points $u$ and $v$, if $W(c_1,z)$ is the common part of $W(c_1,u)$ and $W(c_1,v)$, $c_1\neq z\neq v$, and either $u\in W(c_1,v)$ or the basic paths $\BT(z,c_1)$ (which is the reverse of $W(c_1,z)$), $W(z,u)$, and $W(z,v)$ go out of $z$ in this order clockwise, then we declare $u\prec v$.
It follows that $u\prec v$ for any two basic points $u$ and $v$ such that $\dn(u)<\dn(v)$ in $P$.
The greatest point in this order is $c_h$.
Figure~\ref{fig:same-side} illustrates how the basic tree might appear.

Let $a\in A(S)$.
Following \cite{ST14}, we call a basic point $v$ \emph{special} for $a$ if the following holds:
\begin{itemize}
\item $a\leq v$ in $P$,
\item $a\not\leq u$ in $P$ whenever $u\in W(c_1,v)$ and $u\neq v$.
\end{itemize}
Let $\Spec(a)$ denote the set of all basic points $v$ that are special for $a$ (see Figure~\ref{fig:same-side}).
Any two distinct points $v,v'\in\Spec(a)$ satisfy $v'\notin W(c_1,v)$ and $v\notin W(c_1,v')$, although they may be comparable in $P$.
The set $\Spec(a)$ inherits the order $\prec$ from $\BT$.
Since $\up(a)\in\Spec(a)$, we have $\Spec(a)\neq\emptyset$.
In fact, $\up(a)$ is the greatest point in the order $\prec$ on $\Spec(a)$.

\begin{figure}[t]
\begin{tikzpicture}[double distance=1.2pt,>=latex]
\tikzstyle{every node}=[circle,draw,fill=white,minimum size=4pt,inner sep=0pt]
\tikzstyle{every label}=[rectangle,draw=none,fill=none,minimum size=0pt,inner sep=2pt]
\fill[black!25] (6,0)--(6.5,1)--(6.5,2.1)--(7,3)--(8.5,2)--(9.5,1)--(8,0)--cycle;
\node[label=below:$c_1$] (c1) at (1,0) {};
\node[label=below:$c_2$] (c2) at (2,0) {};
\node[label=below:$c_3$] (c3) at (3,0) {};
\node[label=below:$c_4$] (c4) at (4,0) {};
\node[label=below:$c_5$] (c5) at (5,0) {};
\node[label=below:$c_6$] (c6) at (6,0) {};
\node[label=below:$c_7$] (c7) at (7,0) {};
\node[label=below:$c_8$] (c8) at (8,0) {};
\node[label=below:$c_9$] (c9) at (9,0) {};
\node[label=below:$c_{10}$] (c10) at (10,0) {};
\node[label=below:$c_{11}$] (c11) at (11,0) {};
\node[label=below:$c_{12}$] (c12) at (12,0) {};
\node (u1) at (2,1) {};
\node (u2) at (1.85,2) {};
\node[label=above left:$u_1$] (u3) at (2,3) {};
\node (u4) at (3,1.5) {};
\node[label=right:$u_2$] (u8) at (3.2,2.6) {};
\node (u5) at (4.2,1.1) {};
\node (u6) at (4.25,2.2) {};
\node[label=above:$u_3$] (u7) at (4.8,3.5) {};
\node (u9) at (6.5,1) {};
\node[label={[label distance=2pt]above:$u_4$}] (u10) at (8.5,2) {};
\node (u11) at (9.5,2.5) {};
\node (v1) at (5.25,1.25) {};
\node (v2) at (5.5,2.5) {};
\node[fill=black,label={[label distance=1pt]above:$a$}] (a) at (3.3,3.7) {};
\node[label=above left:$b$] (b) at (7.5,1.5) {};
\node[label=left:$x$] (x) at (6.5,2.1) {};
\node[label=above right:$y$] (y) at (9.5,1) {};
\node[label=above right:$t$] (t) at (7,3) {};
\node[rectangle,draw=none,fill=none] at (7.75,0.75) {$\cgR(t,x,y)$};
\draw[double,->] (c1)--(c2);
\draw[double,->] (c2)--(c3);
\draw[double,->] (c3)--(c4);
\draw[double,->] (c4)--(c5);
\draw[double,->] (c5)--(c6);
\draw[double,->] (c6)--(c7);
\draw[double,->] (c7)--(c8);
\draw[double,->] (c8)--(c9);
\draw[double,->] (c9)--(c10);
\draw[double,->] (c10)--(c11);
\draw[double,->] (c11)--(c12);
\draw[double,->] (c2)--(u1);
\draw[double,->] (u1)--(u2);
\draw[double,->] (u2)--(u3);
\draw[double,->] (u1)--(u4);
\path[->] (u4) edge (u6);
\draw[double,->] (c4)--(u5);
\draw[double,->] (u5)--(u6);
\draw[double,->] (u6)--(u7);
\draw[double,->] (u4)--(u8);
\path[->] (a) edge (u8);
\path[->] (a) edge (u3);
\path[->] (a) edge (u7);
\path[->] (u7) edge (t);
\draw[double,->] (c5)--(v1);
\draw[double,->] (v1)--(v2);
\draw[double,->] (v2)--(t);
\draw[double,->] (c6)--(u9);
\draw[double,->] (u9)--(x);
\draw[double,->] (u9)--(b);
\draw[double,->] (b)--(u10);
\draw[double,->] (u10)--(u11);
\path[->] (t) edge (x);
\path[->] (t) edge (u10);
\path[->] (u10) edge (y);
\path[->] (y) edge (c11);
\draw[double,->] (c8)--(y);
\end{tikzpicture}
\caption{Illustration for the concepts introduced in Section~\ref{sec:same-side}: arrows on the cover graph edges point according to the increasing direction of $P$; the basic tree is marked by empty points and double edges; $\Spec(a)=\{u_1\prec u_2\prec u_3\prec t\prec x\prec u_4\prec y\prec c_{11}\}$; $\Spec'(a)=\{u_1\prec x\prec y\prec c_{11}\}$.}
\label{fig:same-side}
\end{figure}
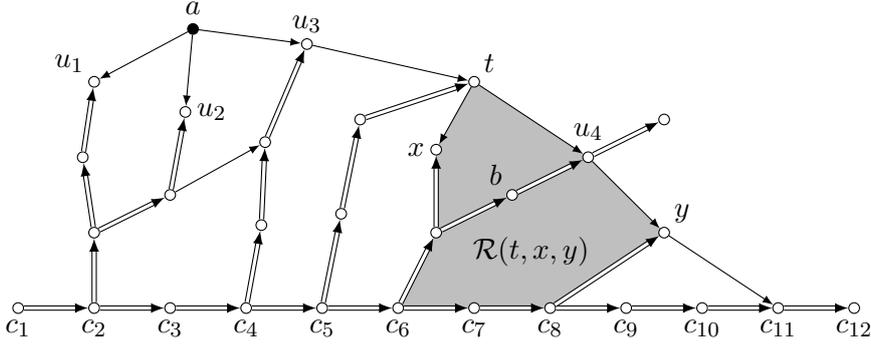

The following straightforward property is stated for emphasis.

\begin{observation}
If\/ $a\in A(S)$, $b\in B(S)$, and\/ $a\leq b$ in\/ $P$, then there is\/ $v\in\Spec(a)$ with\/ $v\in W(c_1,b)$ and\/ $v\leq b$ in\/ $P$.
\end{observation}

Now, suppose $a\in A(S)$, $x,y\in\Spec(a)$, and $x\prec y$ in $\BT$.
Every point $t\in P$ such that $a\leq t\leq x,y$ in $P$ and $t$ is the only common point of $W(t,x)$ and $W(t,y)$ gives rise to a region in the plane, denoted by $\cgR(t,x,y)$, whose boundary is the simple closed curve formed by the paths $\BT(x,y)$, $W(t,x)$, and $W(t,y)$.
The boundary of $\cgR(t,x,y)$ traversed counterclockwise goes from $t$ to $x$ along $W(t,x)$, then to $y$ along $\BT(x,y)$, and then back to $t$ along $W(t,y)$ in the reverse direction (see Figure~\ref{fig:same-side}).
It is possible that $t=x$ or $t=y$.
If $a\neq t$, then $a$ lies either in the interior or in the exterior of $\cgR(t,x,y)$, but not on the boundary.
Any region of the form $\cgR(t,x,y)$ where $x,y\in\Spec(a)$ and $a\leq t$ in $P$ is called an \emph{$a$-region}.
(For such an $a$-region, it is \emph{not} required that $t$ lies on the paths $W(a,x)$ and $W(a,y)$.)

Let $T_0$ denote the subset of $S$ consisting of all pairs that are left-safe with respect to $S$.
By Lemma~\ref{lem:left-safe}, we have $\dim(T_0)\leq d_1$.
To complete the proof, we will show that $\dim(S-T_0)\leq\binom{2k-1}{2}(4k-4)d_1$ by partitioning $S-T_0$ into reversible subsets of the form $T(h_1,h_2,n,\gamma)$, where the four parameters are integers with $0\leq h_1<h_2\leq 2k-2$, $0\leq n\leq 4k-5$, and $1\leq\gamma\leq d_1$.
Membership in these sets will be determined in two stages.
In the first stage, we will partition $S-T_0$ into subsets of the form $T(h_1,h_2)$ where $0\leq h_1<h_2\leq 2k-2$.
In the second stage, for each pair $(h_1,h_2)$ with $0\leq h_1<h_2\leq 2k-2$, we will further partition $T(h_1,h_2)$ into reversible subsets of the form $T(h_1,h_2,n,\gamma)$, where $0\leq n\leq 4k-5$ and $1\leq\gamma\leq d_1$.

To describe the first partition, we need some more definitions.
For a basic point $v$, let $h(v)=\length{W(\dn(v),v)}$.
We have $0\leq h(v)\leq 2k-2$ for every basic point $v$, and we have $h(v)=0$ if and only if $v\in C$.
For a point $a\in A(S-T_0)$, let
\[\Spec'(a)=\{u\in\Spec(a)\colon\text{for every $v\in\Spec(a)$, if $v\prec u$ in $\BT$, then $h(v)>h(u)$}\}\]
(see Figure~\ref{fig:same-side}).
The set $\Spec'(a)$ inherits the order $\prec$ from $\Spec(a)$.
The sequence of numbers $h(v)$ for the points $v\in\Spec'(a)$ considered in the order $\prec$ is strictly decreasing.
The first point in $\Spec'(a)$ is the first point in $\Spec(a)$, and the last point in $\Spec'(a)$ is $\up(a)$.

Now, let $(a,b)\in S-T_0$.
Since $(a,b)$ is not left-safe with respect to $S$, there is a point $u\in\Spec(a)$ with $\dn(u)<\dn(b)$ in $P$.
Consequently, there is a point $u\in\Spec'(a)$ with $\dn(u)<\dn(b)$ in $P$ and thus $u\prec b$ in $\BT$.
We also have $b\prec\up(a)\in\Spec'(a)$ and $b\notin\Spec'(a)$.
If follows that there are two points $x,y\in\Spec'(a)$ consecutive in the order $\prec$ on $\Spec'(a)$ such that $x\prec b\prec y$ in $\BT$.
Let $h_1=h(y)$ and $h_2=h(x)$, so that $0\leq h_1<h_2\leq 2k-2$.
We put the pair $(a,b)$ to the set $T(h_1,h_2)$ of the first partition.

To complete the proof, it remains to show that $\dim(T(h_1,h_2))\leq(4k-4)d_1$ for each pair of integers $(h_1,h_2)$ with $0\leq h_1<h_2\leq 2k-2$.
To this end, we fix an arbitrary pair $(h_1,h_2)$ of this form and show that $\dim(T(h_1,h_2))\leq(4k-4)d_1$ by explaining how $T(h_1,h_2)$ can be partitioned into reversible sets of the form $T(h_1,h_2,n,\gamma)$, where $0\leq n\leq 4k-5$ and $1\leq\gamma\leq d_1$.
This task will require some preliminary work.

In the reasoning used above to put a pair $(a,b)$ in $T(h_1,h_2)$, the points $x$ and $y$ have been defined depending on both $a$ and $b$.
Now, however, for any point $a\in A(T(h_1,h_2))$, there is a unique point $y\in\Spec'(a)$ with $h(y)=h_1$, and there is a unique point $x\in\Spec'(a)$ with $h(x)=h_2$.
Let $y(a)$ and $x(a)$ denote these points, respectively, for any $a\in A(T(h_1,h_2))$.

\begin{lemma}
\label{lem:main}
Let\/ $(a,b),(a',b')\in T(h_1,h_2)$ be such that\/ $a'\leq b$ and\/ $\dn(y(a))<\dn(y(a'))$ in\/ $P$.
Then, for every\/ $a$-region of the form\/ $\cgR(t,x(a),y(a))$, there is an\/ $a'$-region of the form\/ $\cgR(t',x(a'),y(a'))$ such that\/ $\length{W(t',x(a'))}<\length{W(t,x(a))}$.
\end{lemma}

\begin{proof}
Fix an $a$-region $\cgR(t,x(a),y(a))$, where $a\leq t$ in $P$.
For simplicity of notation, let $x=x(a)$, $y=y(a)$, $\cgR=\cgR(t,x(a),y(a))$, $x'=x(a')$, and $y'=y(a')$.
We prove the lemma in several steps, establishing the following claims:
\begin{enumerate}
\item\label{item:y'-ext} $y'$ lies in the exterior of $\cgR$;
\item\label{item:b-int} $b$ lies in the interior or on the boundary of $\cgR$;
\item\label{item:a'<=y} $a'\not\leq y$ in $P$;
\item\label{item:a'<=x} if $a'\leq u$ in $P$ and $u\in W(c_1,x)$, then $u=x$;
\item\label{item:a'-int} $a'$ lies in the interior of $\cgR$;
\item\label{item:t'} there is a point $t'$ on $W(t,x)$ such that $a'\leq t\leq x,y'$ in $P$, $t'\neq t$, and $t'$ is the only common point of $W(t',x)$ and $W(t',y')$;
\item\label{item:x'=x} $x'=x$.
\end{enumerate}
The conclusion of the lemma then follows directly from \ref{item:t'} and \ref{item:x'=x}.

For the proof of \ref{item:y'-ext}, suppose $y'$ does not lie in the exterior of $\cgR$.
Since $\dn(y)<\dn(y')$ in $P$, the first edge of the path $W(\dn(y),y')$ is part of the chain $C$ and lies in the exterior of $\cgR$.
Therefore, the path $W(\dn(y),y')$ crosses the boundary of $\cgR$ at some point $u$ other than $\dn(y)$.
It follows that $u\leq x$ or $u\leq y$ in $P$, which contradicts the fact that $\dn(x)\leq\dn(y)<\dn(y')=\dn(u)$ in $P$.

For the proof of \ref{item:b-int}, suppose $b$ lies in the exterior of $\cgR$.
Since $x\prec b\prec y$ in $\BT$, the basic path $W(\dn(b),b)$ enters the interior of $\cgR$ with the first edge that is not common with $W(\dn(x),x)$ nor $W(\dn(y),y)$.
Therefore, it must exit the interior of $\cgR$ through a point $u$ on $W(t,x)$ or $W(t,y)$.
It follows that $a\leq u\leq b$ in $P$, which is a contradiction.

For the proof of \ref{item:a'<=y}, suppose $a'\leq y$ in $P$.
It follows that $a'$ has a special point $u$ on the basic path $W(c_1,y)$.
We have $h(u)\leq h(y)=h_1=h(y')$.
However, we also have $\dn(u)\leq\dn(y)<\dn(y')$ in $P$ and thus $u\prec y'$ in $\BT$.
This is a contradiction with $y'\in\Spec'(a')$.

For the proof of \ref{item:a'<=x}, suppose $a'\leq u$ in $P$, $u\in W(c_1,x)$, and $u\neq x$.
Assume without loss of generality that $u$ is a special point for $a'$.
We have $h(u)<h(x)=h_2=h(x')$ and thus $x'\prec u$ in $\BT$.
We also have $\dn(u)\leq\dn(y)<\dn(y')$ in $P$ and thus $u\prec y'$ in $\BT$, which implies $h(u)>h(y')$, as $y'\in\Spec'(a')$.
However, since $x'$ and $y'$ are consecutive in the order $\prec$ on $\Spec'(a')$, no point $u\in\Spec(a')$ with $x'\prec u\prec y$ in $\BT$ can satisfy $h(x')>h(u)>h(y')$, which is a contradiction.

For the proof of \ref{item:a'-int}, suppose $a'$ does not lie in the interior of $\cgR$.
By \ref{item:b-int}, $b$ is not in the exterior of $\cgR$, so the path $W(a',b)$ intersects the boundary of $\cgR$ at some point $u$.
Since $a'\leq u$ in $P$, it follows from \ref{item:a'<=y} and \ref{item:a'<=x} that $u$ lies on $W(t,x)$.
Thus $a\leq t\leq u\leq b$ in $P$, which is a contradiction.

By \ref{item:a'-int} and \ref{item:y'-ext}, $a'$ is in the interior and $y'$ is in the exterior of $\cgR$, so the path $W(a',y')$ crosses the boundary of $\cgR$.
For the proof of \ref{item:t'}, let $t'$ be the last common point of $W(a',y')$ with the boundary of $\cgR$ in the order along $W(a',y')$.
Since $a'\leq t'$ in $P$, it follows from \ref{item:a'<=y} and \ref{item:a'<=x} that $t'$ lies on $W(t,x)$ and $t'\neq t$.
Furthermore, it follows from the choice of $t'$ that it is the only common point of $W(t',x)$ and $W(t',y')$.
Therefore, $t'$ satisfies all the conditions of \ref{item:t'}.

It remains to prove \ref{item:x'=x}.
Suppose $x'\neq x$.
The fact that $a'\leq t'\leq x$ in $P$ and \ref{item:a'<=x} imply that $x\in\Spec(a')$.
Since $h(x)=h_2=h(x')$ and $x'\in\Spec'(a')$, we have $x'\prec x$ in $\BT$.
Therefore, the first edge of the path $W(c_1,x')$ that is not common with $W(c_1,x)$ lies in the exterior of $\cgR$.
If $x'$ is not in the exterior of $\cgR$, then the path $W(c_1,x')$ crosses $W(t,x)$ or $W(t,y)$ and thus $a\leq t\leq x'$ in $P$.
If $x'$ is not in the interior of $\cgR$, then the path $W(a',x')$ crosses the boundary of $\cgR$ at some point $u$; again, since $a'\leq u$ in $P$, it follows from \ref{item:a'<=y} and \ref{item:a'<=x} that $u$ lies on $W(t,x)$ and thus $a\leq t\leq u\leq x'$ in $P$.
In either case, we have concluded that $a\leq x'$ in $P$.
Consequently, there is a point $v\in\Spec(a)$ on the path $W(c_1,x')$.
It follows that $v\prec x'\prec x$ in $\BT$ and $h(v)\leq h(x')=h(x)$, which is a contradiction to $x\in\Spec'(a)$.
Thus $x'=x$.
\end{proof}

For incomparable pairs $(a,b),(a',b')\in T(h_1,h_2)$, let $(a,b)\to(a',b')$ denote that $a'\leq b$ and $\dn(y(a))<\dn(y(a'))$ in $P$.
For a pair $(a,b)\in T(h_1,h_2)$, let $\ell(a,b)$ denote the greatest integer $s$ for which there is a sequence $\{(a_\alpha,b_\alpha)\}_{\alpha=0}^s\subseteq T(h_1,h_2)$ such that $(a_0,b_0)\to\cdots\to(a_s,b_s)=(a,b)$.

\begin{lemma}
\label{lem:ell}
For every incomparable pair\/ $(a,b)\in T(h_1,h_2)$, the following holds:
\begin{enumerate}
\item\label{item:ell1} $0\leq\ell(a,b)\leq 4k-5$;
\item\label{item:ell2} if\/ $(a,b)\to(a',b')\in T(h_1,h_2)$, then\/ $\ell(a,b)<\ell(a',b')$.
\end{enumerate}
\end{lemma}

\begin{proof}
Fix a sequence $\{(a_\alpha,b_\alpha)\}_{\alpha=0}^s\subseteq T(h_1,h_2)$ such that $(a_0,b_0)\to\cdots\to(a_s,b_s)=(a,b)$, where $s=\ell(a,b)$.
To see \ref{item:ell1}, choose any $a_0$-region of the form $\cgR(t_0,x(a_0),y(a_0))$, and apply Lemma~\ref{lem:main} repeatedly to obtain a sequence of regions $\{\cgR(t_\alpha,x(a_\alpha),y(a_\alpha))\}_{\alpha=0}^s$ such that $\length{W(t_0,x(a_0))}>\cdots>\length{W(t_s,x(a_s))}\geq 0$; these inequalities are possible only when $s\leq\length{W(t_0,x(a_0))}\leq 4k-5$, where the latter inequality follows from the fact that $\dn(x(a_0))<\up(a_0)\leq\up(t_0)$.
To see \ref{item:ell2}, set $(a_{s+1},b_{s+1})=(a',b')$ and observe that the sequence $\{(a_\alpha,b_\alpha)\}_{\alpha=0}^{s+1}$ witnesses $\ell(a',b')\geq s+1$.
\end{proof}

We partition the set $T(h_1,h_2)$ into subsets of the form $T(h_1,h_2,n,\gamma)$, putting every pair $(a,b)\in T(h_1,h_2)$ into the set $T(h_1,h_2,n,\gamma)$ such that $n=\ell(a,b)$ and $\gamma$ is determined as follows:
\begin{itemize}
\item if $h_1>0$, then $\dn(b)\leq\dn(y(a))<\up(a)$ in $P$, so $(a,b)$ is an incomparable pair of $\Ar(i,i+1)$, where $c_i=\dn(y(a))$, and we let $\gamma=\phi_i(a,b)$;
\item if $h_1=0$, then $y(a)=\up(a)$, so $(a,b)$ is an incomparable pair of $\Ar(j-1,j)$, where $c_j=\up(a)$, and we let $\gamma=\phi_{j-1}(a,b)$.
\end{itemize}
It follows that $0\leq n\leq 4k-5$ (by Lemma \ref{lem:ell}~\ref{item:ell1}) and $1\leq\gamma\leq d_1$.
To complete the proof, it remains to show the following.

\begin{lemma}
Every set\/ $T(h_1,h_2,n,\gamma)$ is reversible.
\end{lemma}

\begin{proof}
Suppose not.
Pick an alternating cycle $\{(a_\alpha,b_\alpha)\}_{\alpha=1}^s$ contained in $T(h_1,h_2,n,\gamma)$, where $s\geq 2$.
For $1\leq\alpha\leq s$, since $\ell(a_\alpha,b_\alpha)=n=\ell(a_{\alpha+1},b_{\alpha+1})$, it follows from Lemma \ref{lem:ell}~\ref{item:ell2} that $(a_{\alpha+1},b_{\alpha+1})\not\to(a_\alpha,b_\alpha)$.
This and $a_\alpha\leq b_{\alpha+1}$ in $P$ yield $\dn(y(a_\alpha))\leq\dn(y(a_{\alpha+1}))$ in $P$, for $1\leq\alpha\leq s$.
This implies that there is $c\in C$ with $\dn(y(a_\alpha))=c$ for $1\leq\alpha\leq s$.
If $h_1>0$ and $c_i=c$, then $\{(a_\alpha,b_\alpha)\}_{\alpha=1}^s$ is an alternating cycle in $\Ar(i,i+1)$, and $\phi_i(a_\alpha,b_\alpha)=\gamma$ for $1\leq\alpha\leq s$, which is a contradiction.
If $h_1=0$ and $c_j=c$, then $\{(a_\alpha,b_\alpha)\}_{\alpha=1}^s$ is an alternating cycle in $\Ar(j-1,j)$, and $\phi_{j-1}(a_\alpha,b_\alpha)=\gamma$ for $1\leq\alpha\leq s$, which is again a contradiction.
\end{proof}

We have partitioned $S$ into $O(k^3d_1)$ reversible subsets of the form $T(h_1,h_2,n,\gamma)$ with $0\leq h_1<h_2\leq 2k-2$, $0\leq n\leq 4k-5$, and $1\leq\gamma\leq d_1$, thus proving that $\dim(S)=O(k^3d_1)$.

\section{Opposite-side dangerous pairs}
\label{sec:opposite-side}

In this section, we show that $\dim(S_{LR})=O(k^2d_1+k^3)$ and $\dim(S_{RL})=O(k^2d_1+k^3)$.
We present the argument only for $S_{RL}$, and the argument for $S_{LR}$ is symmetric.
Recall that the set $S_{RL}$ contains only dangerous incomparable pairs $(a,b)$ of $P$ with $\dn(b)>c_1$ and $\up(a)<c_h$ in $P$.
We begin by setting $S=S_{RL}$.
As the argument proceeds, the meaning of $S$ changes, but the ``new'' set $S$ is always a subset of the ``old'' set $S$.
Each time the meaning of $S$ changes, the target upper bound on $\dim(S)$ is adjusted accordingly.

In the argument given thus far, our main emphasis has been on classifying incomparable pairs.
Now, we want to pay attention to comparable pairs.
When $a\in A(S)$, $b\in B(S)$, and $a\leq b$ in $P$, we call $(a,b)$ a \emph{strong comparable pair} if $\up(a)\leq\dn(b)$ in $P$, and we call $(a,b)$ a \emph{weak comparable pair} if $\up(a)>\dn(b)$ in $P$.
When a comparable pair $(a,b)$ is weak, the witnessing path $W(a,b)$ has length at most $4k-5$ and does not intersect the chain $C$.

\begin{lemma}
\label{lem:weak-comp}
For every strict alternating cycle\/ $\{(a_\alpha,b_\alpha)\}_{\alpha=1}^s$ in\/ $S$, where\/ $s\geq 2$, there is at most one index\/ $\alpha\in\{1,\ldots,s\}$ such that\/ $(a_\alpha,b_{\alpha+1})$ is a strong comparable pair.
\end{lemma}

\begin{proof}
Suppose there are two distinct indices $\alpha,\beta\in\{1,\ldots,s\}$ such that the comparable pairs $(a_\alpha,b_{\alpha+1})$ and $(a_\beta,b_{\beta+1})$ are strong.
Without loss of generality, we have $\up(a_\alpha)\leq\up(a_\beta)$ in $P$.
It follows that $a_\alpha<\up(a_\alpha)\leq\up(a_\beta)\leq\dn(b_{\beta+1})<b_{\beta+1}$ in $P$, which contradicts the assumption that the alternating cycle $\{(a_\alpha,b_\alpha)\}_{\alpha=1}^s$ is strict.
\end{proof}

When $T\subseteq S$, let $Q(T)$ denote the set of weak comparable pairs $(a,b)$ such that $a\in A(T)$ and $b\in B(T)$.
By Lemma~\ref{lem:weak-comp}, if $Q(T)=\emptyset$, then there is no strict alternating cycle in $T$ and hence $T$ is reversible.
In particular, we can assume for the rest of this section that $Q(S)\neq\emptyset$.

For every weak comparable pair $(a,b)\in Q(S)$, the choice of the witnessing paths $W(a,\up(a))$ (which ends with a right-edge), $W(\dn(b),b)$ (which starts with a left-edge), and $W(a,b)$ guarantees the following:
\begin{itemize}
\item the common part of $W(a,\up(a))$ and $W(a,b)$ is $W(a,v)$ for some point $v$;
\item the common part of $W(\dn(b),b)$ and $W(a,b)$ is $W(w,b)$ for some point $w$.
\end{itemize}
This yields a region $\cgD(a,b)$ in the plane whose boundary is a simple closed curve formed by the following four paths (see Figure~\ref{fig:opposite-side}):
\begin{itemize}
\item $W(\dn(b),\up(a))$, called the \emph{middle} of $\cgD(a,b)$,
\item $W(v,\up(a))$, called the \emph{bottom} of $\cgD(a,b)$,
\item $W(\dn(b),w)$, called the \emph{top} of $\cgD(a,b)$,
\item $W(v,w)$, called the \emph{right side} of $\cgD(a,b)$.
\end{itemize}
The name for $W(v,w)$ is justified by the assumption (made at the beginning of the proof) that $c_1$ lies on the outer face of the drawing of $G$ and thus $c_h$ lies inside $\cgD(a,b)$.

\begin{figure}[t]
\begin{tikzpicture}[>=latex]
\tikzstyle{every node}=[circle,draw,fill,minimum size=4pt,inner sep=0pt]
\tikzstyle{every label}=[rectangle,draw=none,fill=none,minimum size=0pt,inner sep=2pt]
\fill[black!10] (4,0)--(4.8,-2)--(8.3,-2)--(8.8,-0.3)--(8.4,2.4)--(5.6,2.8)--(2.7,2.5)--(2,0)--cycle;
\fill[black!25] (6,0)--(6,-1.1)--(7.6,-0.9)--(7.7,0.8)--(5.6,1)--(5,0)--cycle;
\node[label=below:$c_1$] (c1) at (1,0) {};
\node[label=below:$c_2$] (c2) at (2,0) {};
\node[label=below:$c_3$] (c3) at (3,0) {};
\node[label=above:$c_4$] (c4) at (4,0) {};
\node[label=below:$c_5$] (c5) at (5,0) {};
\node[label=above:$c_6$] (c6) at (6,0) {};
\node[label=above:$c_7$] (c7) at (7,0) {};
\node[label=below:$a_1$] (a1) at (6,-1.1) {};
\node[label=above:$b_1$] (b1) at (5.6,1) {};
\node[label=below:$a_2$] (a2) at (8.3,-2) {};
\node[label=above:$b_2$] (b2) at (8.4,2.4) {};
\node (u1) at (7.6,-0.9) {};
\node (u2) at (7.7,0.8) {};
\node (u3) at (4,1.6) {};
\node (u4) at (5.9,2) {};
\node (u5) at (4.8,-2) {};
\node (u6) at (8.8,-0.3) {};
\node (u7) at (2.7,2.5) {};
\node (u8) at (5.6,2.8) {};
\node[rectangle,draw=none,fill=none] at (6.8,-0.5) {$\cgD_1$};
\node[rectangle,draw=none,fill=none] at (7.4,1.8) {$\cgD_2$};
\path[thick,->] (c1) edge (c2);
\path[thick,->] (c2) edge (c3);
\path[thick,->] (c3) edge (c4);
\path[thick,->] (c4) edge (c5);
\path[thick,->] (c5) edge (c6);
\path[thick,->] (c6) edge (c7);
\path[->] (a1) edge (c6) edge (u1);
\path[->] (u1) edge (u2);
\path[->] (u2) edge (b1) edge (u4);
\path[->] (c5) edge (b1);
\path[->] (c3) edge (u3);
\path[->] (u3) edge (b1) edge (u4);
\path[->] (a2) edge (u5) edge (u6);
\path[->] (u5) edge (c4);
\path[->] (u6) edge (u2) edge (b2);
\path[->] (c2) edge (u5) edge (u7);
\path[->] (u7) edge (u3) edge (u8);
\path[->] (u8) edge (b2);
\end{tikzpicture}
\caption{Illustration for the concepts used in Section~\ref{sec:opposite-side}: $\cgD_1=\cgD(a_1,b_1)$ and $\cgD_2=\cgD(a_2,b_2)$.}
\label{fig:opposite-side}
\end{figure}
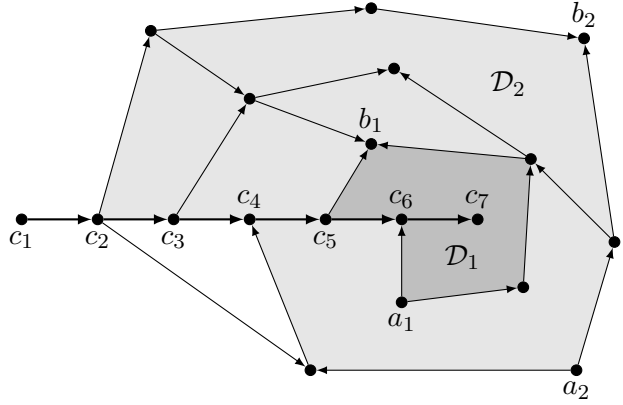

Next, we determine a positive integer $n$ and a sequence of weak comparable pairs $\{(a_i,b_i)\}_{i=1}^{n-1}$ using the following ``greedy'' procedure.
We start by choosing $(a_1,b_1)\in Q(S)$ so as to maximize $\dn(b_1)$.
Then, for $i\geq 2$, after $(a_{i-1},b_{i-1})$ has been determined, we consider all pairs $(a,b)\in Q(S)$ such that the region $\cgD(a_{i-1},b_{i-1})$ is contained in the interior of the region $\cgD(a,b)$, that is, $\dn(b)<\up(a)<\dn(b_{i-1})$ in $P$ and the boundaries of $\cgD(a,b)$ and $\cgD(a_{i-1},b_{i-1})$ are disjoint.
If there is no such pair $(a,b)$, then we set $n=i$ and the construction terminates.
Otherwise, from all such pairs $(a,b)$, we choose $(a_i,b_i)$ so as to maximize $\dn(b_i)$ (see Figure \ref{fig:opposite-side}).

For $1\leq i\leq n-1$, let $u_i=\dn(b_i)$ and $\cgD_i=\cgD(a_i,b_i)$.
Furthermore, let $u_0=c_h$ and $u_n=c_1$.
It follows from the construction that $u_n<u_{n-1}<\cdots<u_1<u_0$ in $P$.
Define functions $M\colon A(S)\to\{1,\ldots,n\}$ and $N\colon B(S)\to\{1,\ldots,n\}$ as follows:
\begin{itemize}
\item for $a\in A(S)$, $M(a)$ is the least positive integer $j$ such that $u_j\leq\up(a)$;
\item for $b\in B(S)$, $N(b)$ is the least positive integer $j$ such that $u_j\leq\dn(b)$.
\end{itemize}
Let $q=8k-9$.

\begin{lemma}
\label{lem:notfar}
If\/ $(a,b)\in Q(S)$, then\/ $N(b)\leq M(a)+q$.
\end{lemma}

\begin{proof}
Let $i=M(a)$ and $j=N(b)$.
The graph $G$ contains a path $W_0$ from $\up(a)$ to $\dn(b)$ which is formed by the bottom, the right side, and the top of $\cgD(a,b)$.
The bottom and the top of $\cgD(a,b)$ have length at most $2k-2$, while the right side has length at most $4k-5$, so $W_0$ has length at most $(2k-2)+(2k-2)+(4k-5)=8k-9$.
In particular, there are at most $8k-10$ points on $W_0$ distinct from $\up(a)$ and $\dn(b)$.

The path $W_0$ starts at the point $\up(a)$, which is in the interior of all $\cgD_{i+1},\ldots,\cgD_{j-1}$, and ends at the point $\dn(b)$, which is in the exterior of all $\cgD_{i+1},\ldots,\cgD_{j-1}$.
Therefore, it must cross the boundaries of the regions $\cgD_{i+1},\ldots,\cgD_{j-1}$ in $j-i-1$ distinct points.
This requires $j-i-1\leq 8k-10$, so that $j\leq i+8k-9=i+q$.
\end{proof}

For every $r\in\{0,\ldots,q-1\}$, let $S(r)=\{(a,b)\in S\colon M(a)\equiv r\pmod{q}\}$.
Thus $S=\bigcup_{r=0}^{q-1}S(r)$.
Since $q=O(k)$, it suffices to show that $\dim(S(r))=O(kd_1+k^2)$ for every $r\in\{0,\ldots,q-1\}$ to complete the proof.
So we fix a value of $r$ and update the meaning of $S$ by setting $S=S(r)$.
For every $m\in\{1,\ldots,n\}$ with $m\equiv r\pmod{q}$, let $S_m=\{(a,b)\in S\colon M(a)=N(b)=m\}$.

\begin{lemma}
We have\/ $\dim(S)\leq d_1+\max\{\dim(S_m)\colon 1\leq m\leq n$, $m\equiv r\pmod{M}\}$.
\end{lemma}

\begin{proof}
Let $d_2=\max\{\dim(S_m)\colon 1\leq m\leq n,$ $m\equiv r\pmod{q}\}$.
For every $m\in\{1,\ldots,n\}$ with $m\equiv r\pmod{q}$, there is a coloring $\psi_m$ that assigns a color $\psi_m(a,b)\in\{d_1+1,\ldots,d_1+d_2\}$ to every incomparable pair $(a,b)\in S_m$ so that the set $\{(a,b)\in S_m\colon\psi_m(a,b)=\gamma\}$ is reversible for every color $\gamma\in\{d_1+1,\ldots,d_1+d_2\}$.

Every pair $(a,b)\in S$ with $N(b)>M(a)$ is an incomparable pair of $\Ar(j-1,j)$, where $c_j=u_{M(a)}$.
Recall that $\phi_{j-1}$ is a coloring of $\Inc(\Ar(j-1,j))$ that uses colors $\{1,\ldots,d_1\}$ and avoids monochromatic alternating cycles.
For every $\gamma\in\{1,\ldots,d_1\}$, let $S(\gamma)$ be the subset of $S$ consisting of all pairs $(a,b)$ such that $N(b)>M(a)$ and $\phi_{j-1}(a,b)=\gamma$, where $c_j=u_{M(a)}$.
For every $\gamma\in\{d_1+1,\ldots,d_1+d_2\}$, let $S(\gamma)$ be the subset of $S$ consisting of all pairs $(a,b)$ such that $N(b)=M(a)$ and $\psi_{M(a)}(a,b)=\gamma$.
Thus $S=\bigcup_{\gamma=1}^{d_1+d_2}S(\gamma)$.

We claim that the set $S(\gamma)$ is reversible for every $\gamma\in\{1,\ldots,d_1+d_2\}$.
Suppose not.
Let $\{(a_\alpha,b_\alpha)\}_{\alpha=1}^s$ be an alternating cycle contained in $S(\gamma)$, where $s\geq 2$.
Suppose there is $\alpha\in\{1,\ldots,s\}$ with $M(a_{\alpha+1})>M(a_\alpha)$.
Then $N(b_{\alpha+1})\geq M(a_{\alpha+1})\geq M(a_\alpha)+q$, as $M(a_{\alpha+1})\equiv M(a_\alpha)\pmod{q}$.
It follows that $(a_{\alpha},b_{\alpha+1})$ is a weak comparable pair, but it violates Lemma~\ref{lem:notfar}.
This shows that $M(a_1),\ldots,M(a_s)$ are all equal.
Let $m=M(a_1)=\cdots=M(a_s)$.
Consequently, $\{(a_\alpha,b_\alpha)\}_{\alpha=1}^s$ is an alternating cycle either in $\Ar(j-1,j)$, when $\gamma\in\{1,\ldots,d_1\}$ and $c_j=u_m$, or in $S_m$, when $\gamma\in\{d_1+1,\ldots,d_1+d_2\}$.
In both cases, this is a contradiction.
\end{proof}

It remains to show that $\dim(S_m)=O(kd_1+k^2)$ for $1\leq m\leq n$.
So we fix a value of $m$ and update the meaning of $S$ by setting $S=S_m$.
It follows that
\begin{itemize}
\item $u_m<\up(a)<u_{m-1}$ in $P$ for every $a\in A(S)$,
\item $u_m\leq\dn(b)<u_{m-1}$ in $P$ for every $b\in B(S)$.
\end{itemize}
If $m=1$, then let $D=\emptyset$.
If $m\geq 2$, then let $D$ denote the set of points $x$ on the boundary of $\cgD_{m-1}$ such that $u_{m-1}\not\leq x$ in $P$.
It follows that $D$ is contained in the union of the bottom and the right side of $\cgD_{m-1}$ with excluded topmost points, and therefore $|D|\leq(2k-2)+(4k-5)-1=O(k)$.

\begin{lemma}
\label{lem:top-right}
Every pair\/ $(a,b)\in Q(S)$ satisfies at least one of the following:
\begin{enumerate}
\item\label{item:top-right1} $\dn(b)=u_m$;
\item\label{item:top-right2} the top of\/ $\cgD(a,b)$ intersects\/ $D$;
\item\label{item:top-right3} the right side of\/ $\cgD(a,b)$ intersects\/ $D$.
\end{enumerate}
\end{lemma}

\begin{proof}
Let $(a,b)\in Q(S)$.
Suppose $\dn(b)\neq u_m$, so that $u_m<\dn(b)<\up(a)<u_{m-1}$ in $P$.
The ``greedy'' construction of the sequence $\{(a_i,b_i)\}_{i=1}^{n-1}$ rejected the pair $(a,b)$, so $m\geq 2$ and $\cgD_{m-1}$ is not contained in the interior of $\cgD(a,b)$.
Since the middle of $\cgD_{m-1}$ lies in the interior of $\cgD(a,b)$, it follows that the top, the bottom, or the right side of $\cgD(a,b)$ intersects the boundary of $\cgD_{m-1}$.

If the bottom of $\cgD(a,b)$ intersects the top of $\cgD_{m-1}$ at a point $x$, then $x\leq\up(a)<u_{m-1}=\dn(b_{m-1})\leq x$ in $P$, which is a contradiction.
If the bottom of $\cgD(a,b)$ intersects the bottom or the right side of $\cgD_{m-1}$ at a point $x$, then $a_{m-1}\leq x\leq\up(a)<u_{m-1}<\up(a_{m-1})$ in $P$, which shows that $\up(a)$ is a better candidate for $\up(a_{m-1})$.
If the top or the right side of $\cgD(a,b)$ intersects the boundary of $\cgD_{m-1}$ at a point $x$ such that $u_{m-1}\leq x$ in $P$, then $\dn(b)<u_{m-1}\leq x\leq b$ in $P$, which shows that $u_{m-1}$ is a better candidate for $\dn(b)$.
We conclude that the top or the right side of $\cgD(a,b)$ intersects $D$.
\end{proof}

Let $R$ be the subset of the maximum chain $C$ consisting of $u_m$ and all points of the form $\dn(x)$ such that $x\in D$ and $u_m<\dn(x)<u_{m-1}$.
It follows that $|R|\leq 1+|D|=O(k)$.
Let $S'=\{(a,b)\in S\colon\dn(b)\in R\}$.
For every pair $(a,b)\in S$ with $\dn(b)=c_i$, we have $(a,b)\in\Inc(\Ar(i,i+1))$ and thus $\dim(\{(a,b)\in S\colon\dn(b)=c_i\})\leq d_1$.
It follows that $\dim(S')\leq|R|\cdot d_1=O(kd_1)$.

\begin{lemma}
\label{lem:right}
For every pair\/ $(a,b)\in Q(S-S')$, the right side of\/ $\cgD(a,b)$ intersects\/ $D$.
\end{lemma}

\begin{proof}
Let $(a,b)\in Q(S-S')$.
Since $b\in B(S-S')$, there is $a'\in A(S-S')$ such that $(a',b)\in S-S'$.
If $\dn(b)=u_m$, then $(a',b)\in S'$ (as $u_m\in R$), which is a contradiction.
Now, suppose the top of $\cgD(a,b)$ intersects $D$ at a point $x$.
Since $x\in W(\dn(b),b)$, we have $\dn(b)=\dn(x)\in R$, so $(a',b)\in S'$, which is a contradiction.
We have excluded the cases \ref{item:top-right1} and \ref{item:top-right2} from Lemma~\ref{lem:top-right}, so the case \ref{item:top-right3} must hold.
\end{proof}

We update the meaning of $S$ once again by setting $S=S-S'$, and we prove that $\dim(S)=O(k^2)$.
Let $\cgX$ denote the family of subsets $X$ of $D$ that are downward-closed in $D$, that is, such that $y\in X$ whenever $x\in X$, $y\in D$, and $y\leq x$ in $P$.
Every nonempty set $X\in\cgX$ is characterized by the pair of points $(x,y)$ such that $x$ is the topmost point of $X$ on the bottom of $\cgD_{m-1}$ and $y$ is the topmost point of $X$ on the right side of $\cgD_{m-1}$.
It follows that there are at most $(2k-2)(4k-5)$ nonempty sets in $\cgX$, so $|\cgX|\leq(2k-2)(4k-5)+1=O(k^2)$.

For a point $b\in B(S)$, let $\downset b=\{x\in P\colon x\leq b$ in $P\}$.
For every $b\in B(S)$, we have $\downset b\cap D\in\cgX$.
For every $X\in\cgX$, let $S_X=\{(a,b)\in S\colon\downset b\cap D=X\}$.

\begin{lemma}
$Q(S_X)=\emptyset$ for every\/ $X\in\cgX$.
\end{lemma}

\begin{proof}
Suppose there is a pair $(a,b)\in Q(S_X)$.
By Lemma~\ref{lem:right}, the right side of $\cgD(a,b)$ intersects $D$, so there is $x\in D$ such that $a\leq x\leq b$ in $P$.
Since $b\in B(S_X)$, we have $\downset b\cap D=X$, so $x\in X$.
Since $a\in A(S_X)$, there is $b'\in B(S_X)$ such that $(a,b')\in S_X$.
It follows that $\downset b'\cap D=X$ and thus $a\leq x\leq b'$ in $P$, which is a contradiction.
\end{proof}

The last lemma and Lemma~\ref{lem:weak-comp} imply that the set of incomparable pairs $S_X$ is reversible for every $X\in\cgX$.
Since $S=\bigcup_{X\in\cgX}S_X$, it follows that $\dim(S)\leq|\cgX|=O(k^2)$.
This completes the proof that $\dim(S_{RL})=O(k^2d_1+k^3)$.

\section{Connections with graph minors}
\label{sec:connections}

Recently, a number of important results connecting dimension with structural graph theory have been proved \cite{BKY16,JMM+16,JMOW-arxiv,JMT+17,JMW18,MW17,Wal17}.
In particular, the following generalization of Theorem~\ref{thm:streib-trotter} is proved in~\cite{Wal17} (see~\cite{MW17} for an alternative proof and~\cite{JMOW-arxiv,JMW18} for further extensions).

\begin{theorem}
\label{thm:minors}
For every positive integer\/ $n$, there exists an integer\/ $d$ such that if\/ $P$ is a poset that excludes\/ $\bfk$ and the cover graph of\/ $P$ does not contain\/ $K_n$ as a topological minor, then\/ $\dim(P)\leq d$.
\end{theorem}

Since posets with cover graphs excluding $K_4$ as a topological minor have dimension bounded independently of the height \cite{JMT+17}, Theorem~\ref{thm:minors} is interesting only for $n\geq 5$.
It is natural to ask whether Theorem~\ref{thm:two-chains} and Conjecture~\ref{con:Sk} can be generalized in the same vein.
To address this question, we need to take a closer look on some properties of interval orders.

An \emph{interval order} is a poset $P$ that admits an \emph{interval representation}, which we define here as an assignment $P\ni x\mapsto(\ell_x,r_x)\subset\setR$ of non-empty open intervals to the points of $P$ such that $x<y$ in $P$ if and only if $r_x\leq\ell_y$.
Such an interval representation of $P$ is \emph{distinguishing} if the endpoints $\ell_x$ and $r_x$ with $x\in P$ are all distinct.
It is easy to see that every interval order admits a distinguishing interval representation.
It is well known that $P$ is an interval order if and only if $P$ excludes $\bftwo+\bftwo$ (the standard example $S_2$) \cite{Fis70}.

Given a distinguishing interval representation of $P$, add two new intervals of the form $(a,\frac{a+b}{2})$ and $(\frac{a+b}{2},b)$ between each pair of consecutive endpoints $a$ and $b$, and call the resulting interval order $Q$.
Then $\dim(P)\leq\dim(Q)$ (as $P$ is a subposet of $Q$) and the cover graph of $Q$ has maximum degree $3$.
This and the fact that there exist interval orders with arbitrarily large dimension \cite{BRT76} lead to the conclusion (observed by Micek and Wiechert~\cite{MW-personal}) that there are interval orders with arbitrarily large dimension and with cover graphs of maximum degree $3$.
Therefore, even for $k=2$, Theorem~\ref{thm:two-chains} and Conjecture~\ref{con:Sk} cannot be generalized to posets that have cover graphs excluding $K_5$ or any other graph with maximum degree greater than $3$ as a topological minor.

By contrast, excluding a minor instead of a topological minor leads to the following observation.

\begin{proposition}
\label{pro:interval}
For every positive integer\/ $n$, there exists an integer\/ $d$ such that if\/ $P$ is an interval order and the cover graph of\/ $P$ does not contain\/ $K_n$ as a minor, then\/ $\dim(P)\leq d$.
\end{proposition}

\begin{proof}
It is proved in~\cite{KT00} that for every interval order $Q$, the interval orders that do not contain $Q$ as a subposet have bounded dimension.
Fix a positive integer $n$.
We construct a poset $Q$ with ground set $\{v_i\colon 1\leq i\leq n\}\cup\{e_{i,j}\colon 1\leq i<j\leq n\}$, where the only cover relations are $v_i<e_{i,j}<v_j$ for $1\leq i<j\leq n$.
It is an interval order witnessed by the interval representation that maps $v_i$ to the interval $(2i-1,2i)$ for $1\leq i\leq n$ and $e_{i,j}$ to the interval $(2i,2j-1)$ for $1\leq i<j\leq n$.
By the result of~\cite{KT00}, there is an integer $d$ such that if $P$ is an interval order with $\dim(P)>d$, then $P$ contains $Q$ as a subposet.
It remains to prove that if $P$ contains $Q$ as a subposet, then the cover graph of $P$ contains $K_n$ as a minor.

Let $P$ be a poset that contains $Q$ as a subposet.
For any two points $a$ and $b$ with $a<b$ in $P$, let $[a,b)_P=\{x\in P\colon a\leq x<b$ in $P\}$ and $(a,b)_P=\{x\in P\colon a<x<b$ in $P\}$.
Let
\[V_j=\bigcup_{i=1}^{j-1}[e_{i,j},v_j)_P\;\cup\;\{v_j\}\;\cup\bigcup_{k=j+1}^n(v_j,e_{j,k})_P\qquad\text{for }1\leq j\leq n.\]
It is clear that the subgraphs of the cover graph of $P$ induced on $V_1,\ldots,V_n$ are connected.
It easily follows from the definition of $Q$ that the sets $V_1,\ldots,V_n$ are mutually disjoint.
A $K_n$ minor in the cover graph of $P$ is obtained by contracting the sets $V_1,\ldots,V_n$ and deleting the vertices not in any of $V_1,\ldots,V_n$.
Indeed, for $1\leq i<j\leq n$, the cover graph of $P$ contains an edge $xe_{i,j}$ connecting $V_i$ and $V_j$, where $x$ is a maximal element in the subset $[v_i,e_{i,j})_P$ of $V_i$ and $e_{i,j}\in V_j$.
\end{proof}

In view of the discussion above, it is natural to make the following conjectures.

\begin{conjecture}
\label{con:minors-two-chains}
For every pair\/ $(k,n)$ of positive integers, there exists an integer\/ $d$ such that if\/ $P$ is a poset that excludes\/ $\bfk+\bfk$ and the cover graph of\/ $P$ does not contain\/ $K_n$ as a minor, then\/ $\dim(P)\leq d$.
\end{conjecture}

\begin{conjecture}
\label{con:minors-Sk}
For every pair\/ $(k,n)$ of positive integers, there exists an integer\/ $d$ such that if\/ $P$ is a poset that excludes the standard example\/ $S_k$ and the cover graph of\/ $P$ does not contain\/ $K_n$ as a minor, then\/ $\dim(P)\leq d$.
\end{conjecture}

Conjecture~\ref{con:minors-Sk} implies Conjecture~\ref{con:minors-two-chains}, because for every pair $(k,n)$ of positive integers, there exists an integer $m$ such that if $P$ is a poset that excludes $\bfk+\bfk$ and the cover graph of $P$ does not contain $K_n$ as a (topological) minor, then $P$ excludes the standard example $S_m$ (see \cite{MW17}, Theorem~3).
Proposition~\ref{pro:interval} shows that Conjectures \ref{con:minors-two-chains} and~\ref{con:minors-Sk} are true for $k=2$.

We remark that Proposition~\ref{pro:interval}, Conjecture~\ref{con:minors-two-chains}, and Conjecture~\ref{con:minors-Sk} are also stated in \cite{MW17}, but erroneously---with excluded topological minors in place of excluded minors.

\end{document}